\documentclass[final]{siamltex}

\newcommand{\eps}{\varepsilon}
\newcommand{\vr}{\varepsilon}

\title{A Parameter--Uniform Finite Difference Method for Multiscale Singularly
Perturbed Linear Dynamical Systems}

\author{S.~Valarmathi\thanks{Department of Mathematics, Bishop Heber College(Autonomous),
Tiruchirappalli-620 017, Tamil Nadu, India ({\tt
valarmathi07@gmail.com}).}
        \and John J.H.~Miller\thanks{Institute for Numerical
Computation and Analysis, Dublin 2, Ireland
(\texttt{jm@incaireland.org}).}}

\begin{document}

\maketitle

\begin{abstract}
A system of singularly perturbed ordinary differential equations
of first order with given initial conditions is considered.  The
leading term of each equation is multiplied by a small positive
parameter. These parameters are assumed to be distinct and they
determine the different scales in the solution to this problem. A
Shishkin piecewise--uniform mesh is constructed, which is used, in
conjunction with a classical finite difference discretization, to
form a new numerical method for solving this problem. It is proved
that the numerical approximations obtained from this method are
essentially first order convergent uniformly in all of the
parameters. 
\end{abstract}



\pagestyle{myheadings} \thispagestyle{plain}
\markboth{S.~Valarmathi and John J.H.~Miller}{Parameter--Uniform
Finite Difference Method}

\section{Introduction}
We consider the  initial value problem for the singularly perturbed
system of linear first order differential equations
\begin{eqnarray}\label{IVP}
E\vec{u}^{\prime}(t)+A(t)\vec{u}(t)=\vec{f}(t), \;\; t \in (0,T],
\;\; \vec{u}(0)\;\; \mathrm{given}.
\end{eqnarray}
Here $\vec{u}$ is a column $n$-vector, $E$ and $A(t)$ are $n \times
n$ matrices, $E=diag(\vec \eps)$, $\vec \eps= (\eps_1, \; \dots, \;
\eps_n) $ with
 $0 < \eps_i \leq 1$ for all $i=1 \dots n$.
For convenience we assume the ordering
\[\eps_1 < \dots < \eps_n.\]
These $n$ distinct parameters determine the $n$ distinct scales in
this multiscale problem. Cases with some of the parameters
coincident are not considered here. We write the problem in the
operator form
\begin{eqnarray*}
\vec{L}\vec{u}=\vec{f},\;\; \vec{u}(0) \;\; \mathrm{ given },
\end{eqnarray*}
where the operator $\vec L$ is defined by
\[ \vec{L}=ED+A(t)\;\; \rm{and} \;\; D=\frac{d}{dt}.  \]
We assume that, for all $t \in [0,T]$, the components $a_{ij}(t)$ of $A(t)$  satisfy the inequalities \\
\begin{eqnarray}\label{a1} a_{ii}(t) > \displaystyle{\sum_{^{j\neq
i}_{j=1}}^{n}}|a_{ij}(t)| \; \; \rm{for}\;\; 1 \le i \le n, \;\;
\rm{and} \;\;a_{ij}(t) \le 0 \;\; \rm{for} \; \; i \neq
j.\end{eqnarray} We take $\alpha$ to be any number such that
\begin{eqnarray}\label{a2} 0 <\alpha < \displaystyle{\min_{^{t \in (0,1]}_{1 \leq i \leq n}}}(\sum_{j=1}^n
a_{ij}(t)).
\end{eqnarray}
We also assume that $T \ge 2\max_{i}(\eps_i)/\alpha$,
 which ensures that the solution domain contains all of the
 layers. This condition is fulfilled if, for example,  $T \ge 2/\alpha$.
We introduce the norms $\parallel \vec{V} \parallel =\max_{1 \leq k
\leq n}|V_k|$ for any n-vector $\vec{V}$, $\parallel y
\parallel =\sup_{0\leq t\leq T}|y(t)|$ for any
scalar-valued function $y$ and $\parallel \vec{y}
\parallel=\max_{1 \leq k \leq n}\parallel y_{k}
\parallel$ for any vector-valued function $\vec{y}$.
Throughout the paper $C$ denotes a generic positive constant,
which is independent of $t$ and of all singular perturbation and
discretization parameters. Furthermore, inequalities between vectors are understood in the componentwise sense.\\

The plan of the paper is as follows. In the next section both
standard and novel bounds on the smooth and singular components of
the exact solution are obtained. The sharp estimates in Lemma 2.4
are proved by mathematical induction, while an interesting
ordering of the points $t_{i,j}$ is established in Lemma 2.6.
In Section 3 the appropriate piecewise-uniform Shishkin meshes are
introduced, the discrete problem is defined and the discrete
maximum principle and discrete stability properties are
established. In Section 4 an expression for the local truncation
error is found and two distinct standard estimates are stated. In
the final section parameter-uniform estimates for the local
truncation error of the smooth and singular components are
obtained in a sequence of lemmas. The section culminates with
the statement and proof of the parameter-uniform error estimate, which is the main result of the paper.\\

The initial value problems considered here arise in many areas of applied mathematics; see for example \cite{Ath}. Parameter uniform numerical methods for simpler
problems of this kind, when all the singular perturbation parameters are equal, were considered in \cite{HV}. A special case of the present problem with $n=3$ was
considered in \cite{BAIL2008}, which also contains numerical results confirming the theory. For this reason further numerical validation is considered to be
unnecessary. A general introduction to parameter uniform numerical methods is given in \cite{FHMORS} and \cite{MORS}.

\section{Analytical results}
The operator $\vec{L}$ satisfies the following maximum principle
\begin{lemma}\label{max} Let $A(t)$ satisfy (\ref{a1}) and (\ref{a2}).
Let $\vec \psi(t)$ be any function in the domain of $\vec{L}$ such
that $\vec\psi(0)\ge 0.$ Then $\vec{L}\vec{\psi}(t) \geq 0$ for
all $t \in (0,T]$ implies that $\vec{\psi}(t) \geq 0$ for all $t
\in [0,T]$.
\end{lemma}
\begin{proof}Let $i^*, t^*$ be such that $\psi_{i^*}(t^{*})=\min_{i,t}\psi_i(t)$
and assume that the lemma is false. Then $\psi_{i^*}(t^{*})<0$ .
From the hypotheses we have $t^*\neq 0$ and
$\psi^{\prime}_{i^*}(t^*)\leq 0$. Thus
\begin{eqnarray*}(\vec{L}\vec \psi(t^*))_{i^*}=\eps_{i^*}\psi^{\prime}_{i^*}(t^*)+a_{i^*,i^*}(t^*)\psi_{i^*}(t^*)+\sum_{j=1,\;j
\neq i^* }^n a_{i^*,j}(t^*)\psi_j(t^*) \\ <
\psi_i^*(t^*)\sum_{j=1\;j \neq i^* }^n a_{i^*,j}<0.
\end{eqnarray*} which contradicts the assumption and proves the
result for $\vec{L}$. \end{proof}

Let $\tilde{A}(t)$ be any principal sub-matrix of $A(t)$ and
$\vec{\tilde{L}}$ the corresponding operator. To see that any
$\vec{\tilde{L}}$ satisfies the same maximum principle as $\vec{L}$,
it suffices to observe that the elements of $\tilde{A}(t)$ satisfy
\emph{a fortiori} the same inequalities as those of $A(t)$.

We remark that the maximum principle is not necessary for the
results that follow, but it is a convenient tool in their proof.
\begin{lemma}\label{stab} Let $A(t)$ satisfy (\ref{a1}) and (\ref{a2}).
If $\vec \psi(t)$ is any function in the domain of $\vec{L}$
then\[\parallel\vec \psi(t)\parallel\le \max\left\{\parallel\vec
\psi(0)\parallel,\frac{1}{\alpha}\parallel \vec{L}\vec
\psi\parallel\right\},\qquad t\in [0,T]\]
\end{lemma}
\begin{proof}Define the two functions
\[\vec{\theta}^{\pm}(t)=\max\{||\vec{\psi}(0)||,\frac{1}{\alpha}||\vec{L}\vec{\psi}||\}\vec{e}\pm
\vec{\psi}(t), \]where $\vec{e}=(1,\;\dots \;,1)^{\prime}$ is the
unit column vector. Using the properties of $A$ it is not hard to
verify that $\vec{\theta}^{\pm}(0)\geq 0$ and
$\vec{L}\vec{\theta}^{\pm}(t)\geq 0$. It follows from Lemma
\ref{max} that $\vec{\theta}^{\pm}(t)\geq 0$ for all $t \in [0,T]$.
\end{proof}

The Shishkin decomposition of the solution $\vec u\;$ of (1) is
given by $\vec u = \vec v + \vec w\;$ where $\vec v \;$ is the
solution of $\vec{L} \vec v = \vec f\;$ on $(0,T]\;$ with $\vec v(0)
= A^{-1}(0)\vec f(0)\;$ and $\vec w \;$ is the solution of
$\vec{L}\vec w = \vec 0\;$ on $(0,T]\;$ with $\vec w(0) = \vec u(0)
- \vec v(0).\;$ Here $\vec v\; , \vec w $ are, respectively, the
smooth
and singular components of $\vec u\;$.\\
The smooth component $\vec{v}$ of $\vec u$ and its derivatives are
estimated the following lemma, which gives bounds showing the
explicit dependence on the inhomogeneous term and the initial
condition.
\begin{lemma}\label{smooth} Let $A(t)$ satisfy (\ref{a1}) and (\ref{a2}). Then there exists a constant $C$, independent of $\eps, \vec{u}(0)$ and
$\vec{f}$, such that
\[\parallel \vec{v}\parallel \leq C\parallel\vec{f}\parallel, \;
\parallel \vec{v}^{\prime}\parallel \leq
C(\parallel\vec{f}\parallel+\parallel\vec{f}^{\prime}\parallel)\]
 and, for all $1 \leq i \leq n$,  \[\parallel \eps_i v_i^{\prime\prime}\parallel \leq
C(\parallel\vec{f}\parallel+\parallel\vec{f}^{\prime}\parallel)\]
\end{lemma}
\begin{proof}We introduce the two functions
$\vec{\psi}^{\pm}(t)=C||\vec{f}||\vec{e}\pm \vec{v}(t)$ where
$\vec{e}$ is the unit column vector. Noting that
$\vec{v}(0)=A^{-1}(0)\vec{f}(0)$, it is not hard to see that
$\vec{\psi}^{\pm}(0)\geq 0$ and $\vec{L}\vec{\psi}^{\pm}(t)\geq
0$. It follows from Lemma \ref{max} that $\vec{\psi}^{\pm}(t)\geq
0$ for all $t \in [0,T]$ and so $\parallel \vec{v}\parallel \leq
C\parallel\vec{f}\parallel$. To estimate the derivative we now
define the two functions
$\vec{\phi}^{\pm}(t)=C(||\vec{f}||+||\vec{f}^{\prime}||)\vec{e}\pm
\vec{v^{\prime}}(t)$. Since $\vec{v}^{\prime}(0)= 0$ and
$\vec{L}\vec{v}^{\prime}= \vec{f}^{\prime}-A^{\prime}\vec{v}$, it
may be verified that $\vec{\phi}^{\pm}(0)\ge 0$ and
$\vec{L}\vec{\phi}^{\pm}(t)\geq 0$. Again by Lemma \ref{max} we
have $\vec{\phi}^{\pm}(t)\geq 0$, which proves the result.
Finally, differentiating the equation $\eps_i
v_i^{\prime}+(A\vec{v})_i=f_i$ and using the estimates of
$\vec{v}$ and $\vec{v}^{\prime}$, we obtain the required bound on
$\eps_i v_i^{\prime \prime}$ \end{proof}

We define the layer functions $B_{i}, 1 \leq i \leq n$, associated
with the solution $\vec u\;$ by
\[B_{i}(t) = e^{-\alpha t / \eps_i}, \; t \in [0,\infty).\]
The following elementary properties of these layer functions, for
all $1 \leq i < j \leq n$, should be
noted:\\
(i) $B_i(t) < B_j(t)$, for all  $t>0$.\\
(ii) $B_i(s) > B_i(t)$,  for all  $0 \leq s<t < \infty$.\\
(iii) $B_i(0)=1$  and  $0<B_i(t) < 1$  for all   $t>0$.\\
Bounds on the singular component $\vec{w}$ of $\vec{u}$ and its
derivatives are contained in
\begin{lemma}\label{singular} Let $A(t)$ satisfy (\ref{a1}) and
(\ref{a2}).Then there exists a constant $C,$ such that, for each
$t \in [0,T]$ and $i=1,\; \dots , \; n$,
\[\left|w_i(t)\right| \;\le\; C B_{n}(t),\;\;
\left|w_i^\prime(t)\right| \;\le\; C\sum_{q=i}^n
\frac{B_{q}(t)}{\eps_q},\;\; \left|\eps_i
w_i^{\prime\prime}(t)\right| \;\le\; C\sum_{q=1}^n
\frac{B_{q}(t)}{\eps_q}.\]
\end{lemma}

\begin{proof}First we obtain the bound on $\vec{w}$. We define the two
functions $\vec{\psi}^{\pm}=CB_n\vec{e} \pm \vec{w}$. Then clearly
$\vec{\psi}^{\pm}(0) \geq 0$ and $L\vec{\psi}^{\pm}=CL(B_n\vec{e})$.
Then, for $i=1,\dots, n$, $(L\vec{\psi}^{\pm})_i
=C(\sum_{j=1}^{n}a_{i,j}-\alpha\frac{\eps_i}{\eps_n})B_n >0$. By
Lemma \ref{max} $\vec{\psi}^{\pm}\geq 0$, which leads to the
required bound on $\vec{w}$.

To establish the bound on $\vec{w}^{\prime}$ we begin with the
$n^{th}$ equation in $\vec{L}\vec{w}=0$, namely
\[ \eps_nw_n^\prime + a_{n,1}w_1 +\dots +a_{n,n}w_n=0, \] from
which the bound for $i=n$ follows. We now bound $w_i^{\prime}$ for
$1 \le i \le n-1$. We define $\vec{p}=(w_1, \dots,w_{n-1})$ and,
taking the first $n-1$ equations satisfied by $\vec{w}$, we get
\[\tilde{A}\vec{p}=\vec{g},\] where $\tilde{A}$ is the matrix
obtained from $A$ by deleting the last row and column and the
components of $\vec{g}$ are $g_k=-a_{k,n}w_n$ for $1 \le k \le n-1$.
Using the bounds already obtained for $\vec{w}$ we see that
$\vec{g}$ is bounded by $CB_n (t)$ and its derivative by $C\frac{B_n
(t)}{\eps_n}$. The initial condition for $\vec{p}$ is
$\vec{p}(0)=\vec{u}(0)-\vec{u}^0(0)$, where $\vec{u}^0$ is the
solution of the reduced problem $\vec{u}^0=A^{-1}\vec{f}$, and is
therefore bounded by
$C(\parallel\vec{u}(0)\parallel+\parallel\vec{f}(0)\parallel)$.
Decomposing $\vec{p}$ into smooth and singular components we get
\[\vec{p}=\vec{q}+\vec{r}, \;\;\ \vec{p}^{\prime}=\vec{q}^{\prime}+\vec{r}^{\prime}.\]
Applying Lemma \ref{smooth}  to $\vec{q}$, from the bounds on the
inhomogeneous term $\vec{g}$ and its derivative $\vec{g}^{\prime}$,
we conclude that $\parallel\vec{q}^{\prime}(t)\parallel \leq
C\frac{B_n(t)}{\eps_n}$. We now use mathematical induction. We
assume that Lemma \ref{singular} is valid for all systems with $n-1$
equations. Then Lemma \ref{singular} applies to $\vec{r}$ and so,
for $i=1, \dots, n-1$,
\[
|r^{\prime}_{i}(t) | \leq
C(\frac{B_{i}(t)}{\eps_i}+\dots+\frac{B_{n-1}(t)}{\eps_{n-1}}).\]
Combining the bounds for $q_i$ and $r_i$ we obtain
\[|p^{\prime}_i(t)|\leq
C(\frac{B_{i}(t)}{\eps_i}+\dots+\frac{B_{n}(t)}{\eps_n}).\]
Recalling the definition of $\vec{p}$ this is the same as
\[|w^{\prime}_i(t)|\leq
C(\frac{B_{i}(t)}{\eps_i}+\dots+\frac{B_{n}(t)}{\eps_n}).\] We have
thus proved that Lemma \ref{singular} holds for our system with $n$
equations. Since Lemma \ref{singular} is true for a system with one
equation,  we conclude by mathematical induction that it is true for
any system of $n>1$ equations.

Finally, to estimate the second derivative, we differentiate the
$i^{th}$ equation of the system $\vec{L}\vec{w}=0$ to get
\[\eps_{k}w_{i}^{\prime\prime}=-(A\vec{w}_i^{\prime} + A^{\prime}\vec{w})_i\]
and we see that the bound on $w_{i}^{\prime\prime}$ follows easily
from the bounds on $\vec{w}$ and $\vec{w}^{\prime}$. \end{proof}
\begin{definition}
For each $1 \leq i \neq j \leq n$ we define the point $t_{i,j}$ by
\begin{equation}\label{t0}\frac{B_i(t_{i,j})}{\varepsilon_i}=\frac{B_j(t_{i,j})}{\varepsilon_j}. \end{equation}
\end{definition}
In the next lemma it is shown that these points exist, are
uniquely defined and have an interesting ordering. Sufficient
conditions for them to lie in the domain $[0,T]$ are also
provided.

 \begin{lemma}\label{ts} For all $i,j$ with $1 \leq i < j \leq n$ the points
$t_{i,j}$ exist, are uniquely defined and satisfy the following
inequalities
\begin{equation}\label{t1}
\vr_i^{-1}B_{i}(t) > \vr_j^{-1}B_{j}(t)\qquad t \in [0,t_{ij})
\end{equation}
and
\begin{equation}\label{t2}
\vr_i^{-1}B_{i}(t) < \vr_j^{-1}B_{j}(t)\qquad t \in (t_{ij},\infty).
\end{equation}
In addition the following ordering holds
\begin{equation}\label{t3}t_{i,j}< t_{i+1,j}, \; \mathrm{if} \;\; i+1<j \;\; \mathrm{and} \;\; t_{i,j}<
t_{i,j+1}, \;\; \mathrm{if} \;\; i<j \end{equation} \\
and
\begin{equation}\label{t4} \eps_i \leq \eps_{j}/2 \;\; \mathrm{implies \;\; that}
\;\; t_{ij} \in (0,T] \;\; \mathrm{for \;\; all} \;\; i<j.
\end{equation}
\end{lemma}
\begin{proof} Existence, uniqueness, (\ref{t1}) and (\ref{t2}) all follow
from the observation that for $i<j$ we have $\eps_i<\eps_j$ and
the ratio of the two sides of (\ref{t0}), namely
\[\frac{B_{i}(t)}{\varepsilon_i}\frac{\varepsilon_j}{B_{j}(t)}=
\frac{\eps_j}{\eps_i}\exp{(-\alpha
t(\frac{1}{\eps_i}-\frac{1}{\eps_j}))},\] is monotonically
decreasing from $\frac{\varepsilon_j}{\varepsilon_i} >1$ to $0$ as
$t$ increases
from $0$ to $\infty$.\\

Rearranging (\ref{t0}) gives
\[t_{i,j}=\frac{\ln(\frac{1}{\eps_i})-\ln(\frac{1}{\eps_j})}{\alpha(\frac{1}{\eps_i}-\frac{1}{\eps_j})}
.\] Writing $\eps_k = \exp(-p_k)$ for some $p_k > 0$ and all $k$
gives
\[t_{i,j}=\frac{p_i -p_j}{\alpha(\exp{p_i} -\exp{p_j})}.\] The
inequality $t_{i,j}< t_{i+1,j}$ is equivalent to
\[\frac{p_i -p_j}{\exp{p_i} -\exp{p_j}}<\frac{p_{i+1} -p_j}{\exp{p_{i+1}} -\exp{p_j}}, \]
which can be written in the form
\[(p_{i+1}-p_j)\exp(p_i-p_j)+(p_{i}-p_{i+1})-(p_{i}-p_j)\exp(p_{i+1}-p_j)>0. \]
With $a=p_i-p_j$ and $b=p_{i+1}-p_j$ it is not hard to see that
$a>b>0$ and $a-b=p_i-p_{i+1}$. Moreover, the previous inequality is
then equivalent to \[\frac{\exp{a}-1}{a}>\frac{\exp{b}-1}{b}, \]
which is true because $a>b$ and proves the first part of
(\ref{t3}). The second part is proved by a similar argument.\\

Finally, to prove (\ref{t4}) it suffices to rearrange (\ref{t0})
in the form
\[t_{i,j}=\frac{\ln(\frac{\eps_j}{\eps_i})}{\alpha(\frac{1}{\eps_i}-\frac{1}{\eps_j})}.
\] Since $T > \frac{2}{\alpha}$ and $\eps_i \leq \frac{\eps_j}{2}$ it follows that
$\ln(\frac{\eps_j}{\eps_i}) \leq \frac{\eps_j}{\eps_i}$ and $t_{i,j}
\in (0,T]$. \end{proof}

\section{The discrete problem}
We construct a piecewise uniform mesh with $N$ mesh-intervals and
mesh-points $\{t_i\}_{i=0}^N$ by dividing the interval $[0,T]$ into
$n+1$ sub-intervals as follows
\[
[0,T]=[0,\sigma_1]\cup(\sigma_1,\sigma_2]\cup\dots(\sigma_{n-1},\sigma_n]\cup(\sigma_n,T]\]
Then, on the sub-interval $[0,\sigma_1]$, a uniform mesh with
$\frac{N}{2^n}$ mesh-intervals is placed, and similarly on
$(\sigma_i, \sigma_{i+1}], 1 \leq i \leq n-1$,  a uniform mesh with
$ \frac{N}{2^{n-i+1}}$ mesh-intervals  and  on  $(\sigma_n,T]$ a
uniform mesh with $\frac{N}{2}$
 mesh-intervals. In practice it is convenient to take $N=2^{n}k$ where $k$
 is some positive power of 2. The $n$ transition
points between the uniform meshes are defined by
\[\sigma_{i}=\min\{\frac{\sigma_{i+1}}{2},\frac{\eps_i}{\alpha}\ln
N\}\]for $i=1,\dots,n-1$ and
\[\sigma_{n}=\min\{\frac{T}{2},\frac{\eps_n}{\alpha}\ln N\}. \]
Clearly \[ 0<\sigma_1 < \dots < \sigma_n \leq \frac{T}{2}. \]This
construction leads to a class of $2^n$ piecewise uniform Shishkin
 meshes $M_{\vec{b}}$, where $\vec b$ denotes an $n$--vector with
$b_i=0$ if $\sigma_i=\frac{\sigma_{i+1}}{2}$ and $b_i=1$ otherwise.
Writing $\delta_j=t_j-t_{j-1}$ we remark that, on any $M_{\vec b}$,
we have
\begin{equation}\label{geom1}\delta_j \leq
CN^{-1}, \;\;\; 1 \leq j \leq N \end{equation}
 and
\begin{equation}\label{geom2}\sigma_i \leq C \eps_i \ln N, \; \;\; 1 \leq i \leq n. \end{equation}
On any $M_{\vec b}$ we now consider the discrete solutions defined
by the backward Euler finite difference scheme
\[ED^{-}\vec{U} +A(t)\vec{U}=\vec{f},  \qquad \vec{U}(0)=\vec{u}(0), \]
or  in operator form
\[\vec{L}^N \vec{U} =\vec{f},  \qquad  \vec{U}(0)=\vec{u}(0), \] where
\[\vec{L}^N=ED^{-}+A(t)\]
and $D^{-}$ is the backward difference operator
\[
D^-\vec{U}(t_j)= \frac {\vec{U}(t_j)-\vec{U}(t_{j-1})}{\delta_j}.
\]
We have the following discrete maximum principle analogous to the
continuous case.
\begin{lemma}\label{dmax} Let $A(t)$ satisfy (\ref{a1}) and (\ref{a2}).
Then, for any mesh function $\vec \Psi$, the inequalities $\vec
{\Psi}(0)\;\ge\;\vec 0 \;\rm{and}\;\vec{L}^N
\vec{\Psi}(t_j)\;\ge\;\vec 0\;$ for $1\;\le\;j\;\le\;N,\;$ imply
that $\;\vec \Psi(t_j)\ge \vec 0\;$ for $0\;\le\;j\;\le\;N.\;$
\end{lemma}

\begin{proof} Let $i^*, j^*$ be such that
$V_{i^*}(t_{j^{*}})=\min_{i,j}V_i(t_j)$ and assume that the lemma
is false. Then $V_{i^*}(t_{j^{*}})<0$ . From the hypotheses we
have $j^*\neq 0$ and $V_{i^*} (t_{j^*})-V_{i^*}(t_{j^*-1})\leq 0$.
Thus
\begin{eqnarray*}(\vec{L^N}\vec{V}(t_{j^*}))_{i^*}=
\eps_{i^*}\frac{V_{i^*}(t_{j^*})-V_{i^*}(t_{j^*-1})}{\delta_{j^*}}+a_{i^*,i^*}(t_{j^*})V_{i^*}(t_{j^*})+\sum_{k=1\;k
\neq i^* }^n a_{i^*,k}(t_{j^*})V_{k}(t_{j^*}) \\ <
V_{i^*}(t_{j^*})\sum_{k=1\;k \neq i^* }^n a_{i^*,k}<0,
\end{eqnarray*} which contradicts the assumption, as required. \end{proof}

An immediate consequence of this is the following discrete stability
result.
\begin{lemma}\label{dstab} Let $A(t)$ satisfy (\ref{a1}) and (\ref{a2}).
Then, for any mesh function $\vec \Psi $,
\[\parallel\vec \Psi(t_j)\parallel\;\le\;\max\left\{\parallel\vec \Psi(0)\parallel,\frac{1}{\alpha}\parallel
\vec{L}^N\vec \Psi\parallel\right\}, 0\leq j \leq N \]
\end{lemma}
\begin{proof} Define the two functions
\[\vec{\Theta}^{\pm}(t)=\max\{||\vec{\Psi}(0)||,\frac{1}{\alpha}||\vec{L^N}\vec{\Psi}||\}\vec{e}\pm
\vec{\Psi}(t)\]where $\vec{e}=(1,\;\dots \;,1)$ is the unit vector.
Using the properties of $A$ it is not hard to verify that
$\vec{\Theta}^{\pm}(0)\geq 0$ and
$\vec{L^N}\vec{\Theta}^{\pm}(t_j)\geq 0$. It follows from Lemma
\ref{dmax} that $\vec{\Theta}^{\pm}(t_j)\geq 0$ for all $0\leq j
\leq N$.
\end{proof}
\section{The local truncation error}
From Lemma \ref{dstab}, we see that in order to bound the error
$\parallel\vec{U}-\vec{u}\parallel$ it suffices to bound
$\vec{L}^N(\vec{U}-\vec{u})$. But this expression satisfies
\[
\vec{L}^N(\vec{U}-\vec{u})=\vec{L}^N(\vec{U})-\vec{L}^N(\vec{u})=
\vec{f}-\vec{L}^N(\vec{u})=\vec{L}(\vec{u})-\vec{L}^N(\vec{u})\]
\[=(\vec{L}-\vec{L}^N)\vec{u} =-E(D^- -D)\vec{u},\] which is the local
truncation of the first derivative. We have
\[E(D^- -D)\vec{u}
=E(D^- -D)\vec{v}+E(D^- -D)\vec{w}
\] and so, by the triangle inequality, \begin{equation}\label{triangleinequality}\parallel \vec{L}^N(\vec{U}-\vec{u})\parallel \leq \parallel
E(D^- -D)\vec{v}\parallel+\parallel E(D^- -D)\vec{w}\parallel.
\end{equation} Thus, we can treat the smooth and singular components
of the local truncation error separately. In view of this we note
that, for any smooth function $\psi$, we have the following two
distinct estimates of the local truncation error of its first
derivative
\begin{equation}\label{lte1}
|(D^- -D)\psi(t_j)|\le2\max_{s\in
I_j}|\psi^{\prime}(s)|\qquad\qquad\;\;
\end{equation}
and
\begin{equation}\label{lte2}
|(D^- -D)\psi(t_j)|\le \frac{\delta_j}{2}\max_{s\in
I_j}|\psi^{\prime\prime}(s)|,
\end{equation}
where $I_j=[t_{j-1},t_j]$.

\section{Error estimate}
We now establish the error estimate by generalizing the approach
based on Shishkin decompositons used in \cite{BAIL2008}. For a
reaction-diffusion boundary value problem in the special case $n=2$
a parameter uniform numerical method was analyzed in \cite{MS} by a
similar technique and in the general case in
\cite{LM} using discrete Green's functions.\\

We estimate the smooth component of the local truncation error in
the following lemma.
\begin{lemma} \label{smootherror} Let $A(t)$ satisfy (\ref{a1}) and (\ref{a2}). Then, for
each $i=1,\; \dots ,\; n$ and $j=1, \;\dots,\; N$, we have
\[ |\eps_{i}(D^- -D)v_i(t_j)|\leq C N^{-1}.\]
\end{lemma}

\begin{proof} Using (\ref{lte2}), Lemma \ref{smooth} and (\ref{geom1}) we
obtain
\[|\eps_{i}(D^{-}-D)v_i(t_j)| \leq C\delta_j
\max_{s \epsilon I_j}|\varepsilon_i v_i^{\prime\prime}(s)| \\
 \leq  C\delta_j \\
  \leq  C N^{-1}
\] as required.
\end{proof}

For the singular component we obtain a similar estimate, but in the
proof we must distinguish between the different types of mesh.  We
need the following preliminary lemmas.
\begin{lemma}\label{est1}  Let $A(t)$ satisfy (\ref{a1}) and (\ref{a2}). Then, for each $i=1,\;
\dots ,\; n$ and $j=1, \;\dots,\; N$, on each mesh $M_{\vec{b}}$, we
have the estimate  \[ |\eps_i(D^- -D)w_i(t_j)| \leq
C\frac{\delta_j}{\eps_1}. \]
\end{lemma}
\begin{proof} From (\ref{lte2}) and Lemma \ref{singular},  we have
\[
\begin{array}{lll} |\eps_{i}(D^{-}-D)w_i(t_j)| & \leq &
C\delta_j
\max_{s \epsilon I_j}|\varepsilon_i w_i^{\prime\prime}(s)| \\
 & \leq & C\delta_j\sum_{q=1}^n \frac{B_{q}(t_{j-1})}{\eps_q} \\
 & \leq &  C\frac{\delta_j}{\eps_1}
\end{array}\] as required. \end{proof}

In what follows we make use of second degree polynomials of the
form
\[p_{i;\theta}=\sum_{k=0}^2
\frac{(t-t_{\theta})^k}{k!}w_{i}^{(k)}(t_{\theta}), \] where
$\theta$ denotes a pair of integers separated by a comma.

\begin{lemma}\label{general} Let $A(t)$ satisfy (\ref{a1}) and (\ref{a2}). Then, for each
$i=1,\; \dots ,\; n$, $j=1, \;\dots,\; N$ and $k=1,\;\;\dots,\;\;
n-1$, on each mesh $M_{\vec{b}}$ with $b_k =1$, there exists a
decomposition \[ w_i=\sum_{m=1}^{k+1}w_{i,m}, \] for which we have
the following estimates for each $m$,  $1 \le m \le k$,
\[|\eps_i
w_{i,m}^{\prime}(t)| \leq CB_m(t), \;\;|\eps_i
w_{i,m}^{\prime\prime}(t)| \leq C\frac{B_{m}(t)}{\eps_m}\] and
\[|\eps_i
w_{i,k+1}^{\prime\prime}(t)| \leq
C\sum_{q=k+1}^{n}\frac{B_{q}(t)}{\eps_q}.\]

Furthermore \[ |\eps_i(D^{-}-D)w_i(t_j)| \leq
C(B_{k}(t_{j-1})+\frac{\delta_j}{\eps_{k+1}}).\]
\end{lemma}
\begin{proof} Since $b_k =1$ we have $\eps_{k}\leq \eps_{k+1}/2$,  so
$t_{k,k+1} \in (0,T]$ and we can define the decomposition
\[w_i=\sum_{m=1}^{k+1}w_{i,m},\] where the components of the
decomposition are defined by
\[w_{i,k+1}=\left\{ \begin{array}{ll} p_{i;k,k+1} & {\rm on}\;\;[0,t_{k,k+1})\\
 w_i & {\rm otherwise} \end{array}\right. \]
and for each $m$,  $k \ge m \ge 2$,
\[w_{i,m}=\left\{ \begin{array}{ll} p_{i;m-1,m} & \rm{on} \;\; [0,t_{m-1,m})\\
w_i-\sum_{q=m+1}^{k+1} w_{i,q} & {\rm otherwise}
\end{array}\right. \]
and
\[w_{i,1}=w_i-\sum_{q=2}^{k+1} w_{i,q}\;\; \rm{on} \;\; [0,T]. \]
From the above definitions we note that for each $m$, $1 \leq m
\leq k$,
$w_{i,m}=0 \;\; \rm{on} \;\; [t_{m,m+1},T]$.\\
To establish the bounds on the second derivatives we observe that:

in $[t_{k,k+1},T]$, using Lemma \ref{singular} and $t \geq
t_{k,k+1}$, we obtain
\[|\eps_i w_{i,k+1}^{\prime \prime}(t)| =|\eps_i w_{i}^{\prime \prime}(t)| \leq
C\sum_{q=1}^n \frac{B_q(t)}{\eps_q} \leq C\sum_{q=k+1}^n
\frac{B_q(t)}{\eps_q};\]

in $[0, t_{k,k+1}]$, using Lemma \ref{singular} and $t \leq
t_{k,k+1}$ , we obtain
\[|\eps_i w_{i,k+1}^{\prime \prime}(t)| =|\eps_i w_{i}^{\prime
\prime}(t_{k,k+1})| \leq \sum_{q=1}^{n}
\frac{B_q(t_{k,k+1})}{\eps_q} \leq \sum_{q=k+1}^{n}
\frac{B_q(t_{k,k+1})}{\eps_q} \leq \sum_{q=k+1}^{n}
\frac{B_q(t)}{\eps_q};\]

and for each $m=k, \;\; \dots \;\;,2$, we see that\\

in $[t_{m,m+1},T]$, $w_{i,m}^{\prime \prime}=0;$

in $[t_{m-1,m},t_{m,m+1}]$, using Lemma \ref{singular}, we obtain
\[|\eps_i w_{i,m}^{\prime \prime}(t)| \leq |\eps_i w_{i}^{\prime
\prime}(t)|+\sum_{q=m+1}^{k+1}|\eps_i w_{i,q}^{\prime \prime}(t)|
\leq C\sum_{q=1}^n \frac{B_q(t)}{\eps_q} \leq
C\frac{B_m(t)}{\eps_m};\]

in $[0, t_{m-1,m}]$, using Lemma \ref{singular} and $t \leq
t_{m-1,m}$, we obtain
\[|\eps_i w_{i,m}^{\prime \prime}(t)| =|\eps_i w_{i}^{\prime
\prime}(t_{m-1,m})| \leq C\sum_{q=1}^n \frac{B_q(t_{m-1,m})}{\eps_q}
\leq C\frac{B_m(t_{m-1,m})}{\eps_m} \leq C\frac{B_m(t)}{\eps_m};
\]

in $[t_{1,2},T],\;\; w_{i,1}^{\prime \prime}=0;$

in $[0, t_{1,2}]$, using Lemma \ref{singular}, \[|\eps_i
w_{i,1}^{\prime \prime}(t)| \leq |\eps_i w_{i}^{\prime
\prime}(t)|+\sum_{q=2}^{k+1}|\eps_i w_{i,q}^{\prime \prime}(t)|\leq
C\sum_{q=1}^n \frac{B_q(t)}{\eps_q} \leq C\frac{B_1(t)}{\eps_1}.\]

For the bounds on the first derivatives we observe that for each
$m$, $1 \leq m \leq k $ :

in $[t_{m,m+1},T],\;\; w_{i,m}^{\prime}=0;$

in $[0, t_{m,m+1}]\;\; \int_t^{t_{m,m+1}}\eps_i w_{i,m}^{\prime
\prime}(s)ds= \eps_i w_{i,m}^{\prime}(t_{m,m+1})- \eps_i
w_{i,m}^{\prime}(t)= -\eps_i w_{i,m}^{\prime}(t)$ \\
and so
\[|\eps_i w_{i,m}^{\prime}(t)| \leq \int_t^{t_{m,m+1}}|\eps_i
w_{i,m}^{\prime \prime}(s)|ds \leq
\frac{C}{\eps_m}\int_{t}^{t_{m,m+1}} B_m(s)ds \leq CB_m(t).\]
Finally, since \[|\eps_i(D^{-}-D)w_i(t_j)| \leq
\sum_{m=1}^{k}|\eps_i(D^{-}-D)w_{i,m}(t_j)|+
|\eps_i(D^{-}-D)w_{i,k+1}(t_j)|,\] using  (\ref{lte2}) on the last
term and (\ref{lte1}) on all other terms on the right hand side, we
obtain
\[|\eps_i(D^{-}-D)w_i(t_j)| \leq C(\sum_{m=1}^{k}\max_{s \in
I_j}|\eps_iw_{i,m}^\prime(s)| +\delta_j\max_{s \in
I_j}|\eps_iw_{i,k+1}^{\prime \prime}(s)|).\]  The desired result
follows by applying the bounds on the derivatives in the first part
of this lemma. \end{proof}

%
%
\begin{lemma}\label{est3} Let $A(t)$ satisfy (\ref{a1}) and (\ref{a2}). Then, for each $i=1,\;
\dots ,\; n$ and $j=1, \;\dots,\; N$, on each mesh $M_{\vec{b}}$,
 we have the estimate \[
|\eps_i(D^{-}-D)w_i(t_j)| \leq CB_n(t_{j-1}).\]
\end{lemma}
\begin{proof} From (\ref{lte1}) and Lemma \ref{singular}, for each $i=1,\;
\dots ,\; n$ and $j=1, \;\dots,\; N$, we have
\[
\begin{array}{lll} |\eps_{i}(D^{-}-D)w_i(t_j)| & \leq &
C\max_{s \epsilon I_j}|\varepsilon_i w_i^{\prime}(s)| \\
 & \leq & C\eps_{i}\sum_{q=i}^n \frac{B_{q}(t_{j-1})}{\eps_q}\\
 & \leq &  CB_n(t_{j-1})
\end{array}\] as required.
\end{proof}

Using the above preliminary lemmas on appropriate subintervals we
obtain the desired estimate of the singular component of the local
truncation error in the following lemma.
\begin{lemma}\label{singularerror} Let $A(t)$ satisfy (\ref{a1}) and (\ref{a2}). Then, for each $i=1,\;
\dots ,\; n$ and $j=1, \;\dots,\; N$, we have the estimate
\[ |\eps_{i}(D^- -D)w_i(t_j)| \leq CN^{-1}\ln N. \]
\end{lemma}
\begin{proof}  We consider each subinterval separately.\\
\noindent In the subinterval $(0,\sigma_1]$ we have $\delta_j \leq
CN^{-1}\sigma_1$. On any mesh $M_{\vec{b}}$, using Lemma
\ref{est1}, we get $|\eps_i(D^{-}-D)w_i(t_j)| \leq
CN^{-1}\frac{\sigma_1}{\eps_1}\leq CN^{-1}\ln N $.
\noindent In
the subinterval $(\sigma_1,\sigma_2]$ we have $\delta_j \leq
CN^{-1}\sigma_2.$

On any mesh $M_{\vec{b}}$ with $b_1=0$, we have
$\sigma_2=2\sigma_1$. Using Lemma \ref{est1} we get $|\eps_{i}(D^-
-D)w_i(t_j)|\leq CN^{-1}\frac{\sigma_2}{\eps_1}\leq
CN^{-1}\frac{\sigma_1}{\eps_1}\leq CN^{-1}\ln N$.

On any mesh $M_{\vec{b}}$ with $b_1=1$, we have
$\sigma_1=\frac{\eps_1}{\alpha}\ln N$. Using Lemma \ref{general}
with $k=1$ we get $|\eps_{i}(D^- -D)w_i(t_j)|\leq
C(B_1(\sigma_1)+N^{-1}\frac{\sigma_2}{\eps_2})\leq CN^{-1}\ln N$.

\noindent In a general subinterval $(\sigma_m,\sigma_{m+1}]\; 2 \leq
m \leq n-1,$ we have $\delta_j \leq CN^{-1}\sigma_{m+1}$.

On any mesh $M_{\vec{b}}$ with $b_q=0, \; q=1, \;\dots \;,m$, we
have $\sigma_{m+1}=C\sigma_1$. Using Lemma \ref{est1} we get
$|\eps_{i}(D^- -D)w_i(t_j)|\leq
CN^{-1}\frac{\sigma_{m+1}}{\eps_1}\leq
CN^{-1}\frac{\sigma_1}{\eps_1}\leq CN^{-1}\ln N$.

On any mesh $M_{\vec{b}}$ with $ b_1=1, b_q=0, \; q=2, \;\dots
\;,m$, we have $\sigma_1=\frac{\eps_1}{\alpha}\ln
N,\;\sigma_{m+1}=C\sigma_2$. Using Lemma \ref{general} with $k=1$ we
get $|\eps_{i}(D^- -D)w_i(t_j)|\leq
C(B_1(\sigma_m)+N^{-1}\frac{\sigma_{m+1}}{\eps_2})\leq
C(B_1(\sigma_1)+N^{-1}\frac{\sigma_2}{\eps_2})\leq CN^{-1}\ln N$.

On any mesh $M_{\vec{b}}$ with $ b_k=1,\; b_q=0, \; q=k+1, \;\dots
\;,m$, we have $\sigma_k=\frac{\eps_k}{\alpha}\ln
N,\;\sigma_{m+1}=C\sigma_{k+1}$. Using Lemma \ref{general} with
general $k$ we get $|\eps_{i}(D^- -D)w_i(t_j)|\leq
C(B_k(\sigma_m)+N^{-1}\frac{\sigma_{m+1}}{\eps_{k+1}})\leq
C(B_k(\sigma_k)+N^{-1}\frac{\sigma_{k+1}}{\eps_{k+1}})\leq
CN^{-1}\ln N$.

On any mesh $M_{\vec{b}}$ with $ b_m=1$, we have
$\sigma_m=\frac{\eps_m}{\alpha}\ln N$. Using Lemma \ref{est3} we get
$|\eps_{i}(D^- -D)w_i(t_j)|\leq CN^{-1}B_m(\sigma_m)\leq CN^{-1}\ln
N$.

\noindent In the subinterval $(\sigma_n,T]$ we have $\delta_j \leq
CN^{-1}$.

On any mesh $M_{\vec{b}}$ with $b_q=0, \; q=1, \;\dots \;,n$, we
have $1/\eps_1 \leq C\ln N$. Using Lemma \ref{est1} we get
$|\eps_{i}(D^- -D)w_i(t_j)|\leq CN^{-1}/\eps_1 \leq
 CN^{-1}\ln N$.

On any mesh $M_{\vec{b}}$ with $ b_1=1,\; b_q=0, \; q=2, \;\dots
\;,n$, we have $\sigma_1=\frac{\eps_1}{\alpha}\ln N,\;1/ \eps_2 \leq
C\ln N$. Using Lemma \ref{general} with $k=1$ we get $|\eps_{i}(D^-
-D)w_i(t_j)|\leq C(B_1(\sigma_n)+N^{-1}\ /\eps_2)\leq
C(B_1(\sigma_1)+N^{-1}\ /\eps_2)\leq CN^{-1}\ln N$.

On any mesh $M_{\vec{b}}$ with $ b_k=1,\; b_q=0, \; q=k+1, \;\dots
\;,n,\; 2 \leq k \leq n-1$, we have
$\sigma_k=\frac{\eps_k}{\alpha}\ln N,\;1/\eps_{k+1} \leq C\ln N$.
Using Lemma \ref{general} with general $k$ we get $|\eps_{i}(D^-
-D)w_i(t_j)|\leq C(B_k(\sigma_n)+N^{-1}/ \eps_{k+1})\leq
C(B_k(\sigma_k)+N^{-1}/ \eps_{k+1})\leq CN^{-1}\ln N$.

On any mesh $M_{\vec{b}}$ with $ b_n=1$, we have
$\sigma_n=\frac{\eps_n}{\alpha}\ln N$. Using Lemma \ref{est3} we
get $|\eps_{i}(D^- -D)w_i(t_j)|\leq CN^{-1}B_n(\sigma_n)\leq
CN^{-1}\ln N$.

It is not hard to verify that on each of the $n+1$ subintervals we
have obtained the required estimate for all of the $2^n$ possible
meshes. \end{proof}

 Let $\vec u$ denote the exact solution of
(\ref{IVP}) and $\vec U$ the discrete solution. Then, the main
result of this paper is the following $\eps$-uniform error
estimate
\begin{theorem} Let $A(t)$ satisfy (\ref{a1}) and (\ref{a2}). Then
 there exists a constant $C$ such that \[\parallel\vec{U}-\vec{u}\parallel \leq CN^{-1}\ln
 N,
\] for all $N > 1$
\end{theorem}
\begin{proof} This follows immediately by applying Lemmas
\ref{smootherror} and \ref{singularerror} to
(\ref{triangleinequality}) and using Lemma \ref{dstab}.  \end{proof}
\section*{Acknowledgments}
The first author acknowledges the support of the UGC, New Delhi,
India  under the Minor Research Project-X Plan period.

\end{document}